\documentclass[12pt]{amsart}
\usepackage{amssymb,latexsym}
\usepackage{amsfonts}
\usepackage{amsmath}
\usepackage[colorlinks,linkcolor=blue,anchorcolor=blue,citecolor=blue]{hyperref}
\usepackage{algorithm}
\usepackage{enumerate}
\usepackage{algpseudocode}
\usepackage{verbatim}
\usepackage{graphicx}

\newtheorem{theorem}{Theorem}[section]
\newtheorem{lemma}[theorem]{Lemma}

\newtheorem{conjecture}[theorem]{Conjecture}
\theoremstyle{definition}

\numberwithin{equation}{section}

\begin{document}

\title[Proof of Sondow's conjecture]{A proof of Sondow's conjecture on the Smarandache function}

\author{Xiumei Li}
\address{School of Mathematical Sciences, Qufu Normal University, Qufu, 273165, China}
\email{lxiumei2013@qfnu.edu.cn}

\author{Min Sha}
\address{School of Mathematics and Statistics, University of New South Wales, Sydney, NSW 2052, Australia}
\email{shamin2010@gmail.com}






\maketitle

\begin{abstract}
The Smarandache function of a positive integer $n$, denoted by $S(n)$, is defined to be
the smallest positive integer $j$ such that $n$ divides the factorial $j!$.
In this note, we prove that for any fixed number $k > 1$, the inequality $n^k < S(n)!$ holds for almost all positive integers $n$.
This confirms Sondow's conjecture which asserts that the inequality $n^2 < S(n)!$ holds for almost all positive integers $n$.
\end{abstract}

\section{Introduction.}

In 2006 Sondow \cite{Sondow}
gave a new measure of irrationality
for $e$ (the base of the natural logarithm), that is, for all integers $m$ and $n$ with $n>1$
\begin{equation}\label{eq:irr1}
 \left| e-\frac{m}{n} \right| > \frac{1}{(S(n)+1)!},
\end{equation}
where $S(n)$ is  the smallest positive integer $j$ such that $n$ divides the factorial $j!$.
On the other hand, there is a well-known irrationality measure for $e$ (see, for instance, \cite[Theorem 1]{Davis}):
given any $\epsilon>0$ there exists a positive constant $n(\epsilon)$ such that
\begin{equation}\label{eq:irr2}
 \left| e-\frac{m}{n} \right| > \frac{1}{n^{2+\epsilon}}
\end{equation}
for all integers $m$ and $n$ with $n>n(\epsilon)$.
By contrast, Dirichlet's approximation theorem implies that the inequality
$$
 \left| e-\frac{m}{n} \right| < \frac{1}{n^2}
$$
 is satisfied for infinitely many
integers $m$ and $n$ with $n>1$,
and so it implies that the lower bound in \eqref{eq:irr2} is somehow optimal.
Sondow asserted that \eqref{eq:irr2} is usually stronger than \eqref{eq:irr1} by posing the following conjecture.

\begin{conjecture}[{\cite[Conjecture 1]{Sondow}}]  \label{conjecture}
The inequality $n^2 < S(n)!$ holds for almost all positive integers $n$.
\end{conjecture}

As indicated in \cite{Sondow}, in Conjecture~\ref{conjecture}, $S(n)$ can be replaced by $P(n)$ due to a result of Ivi{\' c}
\cite[Theorem 1]{Ivic}, where $P(n)$ is the largest prime factor of $n$ for $n \ge 2$ (put $P(1)=1$).
By definition, $P(n) \le S(n)$ for any positive integer $n$.

In number theory, $S(n)$ is called the Smarandache function.
This function was studied by Lucas \cite{Lucas} for powers of primes
and then by Neuberg \cite{Neu} and Kempner \cite{Kem} for general $n$.
In particular, Kempner \cite{Kem} gave the first correct algorithm for computing this function.
In 1980 Smarandache \cite{Sma} rediscovered this function.
It is also sometimes called the Kempner function.
This function arises here and there in number theory, as demonstrated in \cite{Sondow}.
Please see \cite{Liu} for a survey on recent results and \cite{HS} for a generalization to several variables.
In addition, the polynomial analogue of the Smarandache function has been applied in \cite{LS2,LS3} and studied in detail in \cite{LS1}.

In this note, we prove a stronger form of Conjecture~\ref{conjecture}.

For any real $k> 1$ and $x > 1$, denote by $N_k(x)$ the number of
positive integers $n$ such that $n \leq x$ and $S(n)! \leq n^k$.

\begin{theorem}\label{thm}
For any fixed number $k > 1$ and any sufficiently large $x$, we have
$$
N_k(x) \le x \exp \left(-\sqrt{2\log x \log\log x} \big(1+ O( \log\log\log x /\log\log x) \big) \right).
$$
\end{theorem}

We remark that the meaning of ``sufficiently large" in Theorem~\ref{thm} depends only on $k$.

From Theorem~\ref{thm}, for any $k > 1$, we have $N_k(x) / x \to 0$ as $x \to \infty$.
This in fact confirms Conjecture~\ref{conjecture} when $k=2$.

Our approach in fact can achieve more.
Let $M(x)$ be the number of
positive integers $n$ such that $n \leq x$ and $S(n)! \leq \exp(n^{1 / \log\log n})$.
Note that, for any fixed $k >1$ and any sufficiently large $n$, we have
$$
n^k < \exp(n^{1 / \log\log n}).
$$

\begin{theorem}\label{thm2}
$M(x) \ll x/ \sqrt{\log x}$.
\end{theorem}

Theorem~\ref{thm2} implies that the inequality $\exp(n^{1 / \log\log n}) < S(n)!$ holds for almost all $n$.

Here we use the big O notation, $O$ and the Vinogradov symbol $\ll$.
We recall that the assertions $f(x)=O(g(x))$ and $ f(x) \ll g(x)$
are both equivalent to the inequality $|f(x)|\le c g(x)$ with some absolute constant $c > 0$
for any sufficiently large $x$.

\section{Proofs of Theorems~\ref{thm} and \ref{thm2}.}

To prove Theorems~\ref{thm} and \ref{thm2} we need the following three lemmas.

\begin{lemma}[{\cite[Theorem 1]{Ivic}}]  \label{lem:Ivic}
For any $x > 1$, denote by $N(x)$ the number of
positive integers $n$ such that $n \le x$ and  $S(n) \neq P(n)$.
Then
$$
 N(x) = x \exp \left(-\sqrt{2\log x \log\log x} \big(1+ O(\log\log\log x /\log\log x) \big) \right).
$$
\end{lemma}

\begin{lemma}[{\cite[Chapter I.0, Corollary 2.1]{Tenenbaum1995}}]  \label{lem:fac}
For any integer  $n \ge 1$, we have
$$
\log n!=n\log n - n + 1 +\theta\log n
$$
with $\theta=\theta_n\in [0,1]$.
\end{lemma}

\begin{lemma}[{\cite[Chapter III.5, Theorem 1]{Tenenbaum1995}}] \label{lem:main}
For any $2\leq y \leq x$, denote by $\Psi(x,y)$ the number of positive integers $n$ such that $n\leq x$ and $P(n) \leq y$.
Then
$$
\Psi(x,y)\ll x\exp\left(-\frac{\log x}{2\log y}\right).
$$
\end{lemma}

We are now ready to prove Theorems~\ref{thm} and \ref{thm2}.

\begin{proof}[Proof of Theorem~\ref{thm}]
We first separate the integers $n$ counted in $N_k(x)$  into two cases
depending on whether $S(n)\neq P(n)$ or $S(n)= P(n)$.
So, we define
\begin{align*}
& N_{k,1}(x)=|\{n\leq x: \, S(n)!\leq n^k, S(n) \neq P(n)\}|, \\
& N_{k,2}(x)=|\{n\leq x: \, S(n)!\leq n^k, S(n)= P(n)\}|.
\end{align*}
Then
\begin{equation}  \label{eq:Nk}
N_k(x)=N_{k,1}(x)+N_{k,2}(x).
\end{equation}

Using Lemma~\ref{lem:Ivic},
we obtain
\begin{equation} \label{eq:Nk1}
\begin{split}
N_{k,1}(x) & \leq N(x) \\
& = x \exp \left(-\sqrt{2\log x \log\log x} \big(1+ O(\log\log\log x /\log\log x) \big) \right).
\end{split}
\end{equation}

We next estimate $N_{k,2}(x)$.
The integers $n$ counted in $N_{k,2}(x)$ can be divided into
the following two cases:
\begin{enumerate}
\item[(i)] $ S(n)!\leq n^k$ and $S(n)= P(n) \leq 5 $;

\item[(ii)] $ S(n)!\leq n^k$ and $S(n)= P(n)\geq 7$.
\end{enumerate}

In case (i) there are at most $12$ possibilities for $n$ by considering $S(n)= P(n)\leq 5 $
(that is, $1, 2, 3, 5, 6, 10, 15, 20, 30, 40, 60, 120$).

For any integer $n$ in case (ii), using Lemma~\ref{lem:fac} we have
\begin{align*}
e \left(\frac{P(n)}{e} \right)^{P(n)}\leq P(n)! = S(n)! \le n^k \le x^k,
\end{align*}
which, together with $P(n)\geq 7$, gives
\begin{equation} \label{eq:Pn}
P(n)\leq 1 + P(n) \log \frac{P(n)}{e} \le   k\log x.
\end{equation}
 So, we obtain
 $$
 N_{k,2}(x) \leq 12 + \Psi(x,k\log x).
 $$
 By Lemma \ref{lem:main},
$$
\Psi(x,k\log x)\ll x \exp\left(-\frac{\log x}{2(\log k+\log\log x) }\right)
$$
when $2\leq k\log x\leq x$.
Thus, for any sufficiently large $x$ we get
\begin{equation}  \label{eq:Nk2}
N_{k,2}(x) \ll x \exp\left(-\frac{\log x}{2(\log k+\log\log x) }\right).
\end{equation}
Finally, combining \eqref{eq:Nk} with \eqref{eq:Nk1} and \eqref{eq:Nk2}, we have
$$
N_k(x) \le x \exp \left(-\sqrt{2\log x \log\log x} \big(1+ O(\log\log\log x /\log\log x) \big) \right)
$$
for any fixed $k > 1$ and any sufficiently large $x$.
This completes the proof.
 \end{proof}

 \begin{proof}[Proof of Theorem~\ref{thm2}]
We use the same approach as in proving Theorem~\ref{thm}.
First, we have
\begin{equation}  \label{eq:M}
M(x) = M_1(x) + M_2(x),
\end{equation}
where
\begin{align*}
& M_1(x)=|\{n\leq x: \, S(n)! \leq \exp(n^{1/ \log\log n}), S(n) \neq P(n)\}|, \\
& M_2(x)=|\{n\leq x: \, S(n)! \leq \exp(n^{1/ \log\log n}), S(n)= P(n)\}|.
\end{align*}
As before, we obtain
\begin{equation}  \label{eq:M1}
\begin{split}
M_1(x) & \leq N(x) \\
& = x \exp \left(-\sqrt{2\log x \log\log x} \big(1+ O(\log\log\log x /\log\log x) \big) \right).
\end{split}
\end{equation}

As in the derivation of \eqref{eq:Pn}, for any integer $n$ counted in $M_2(x)$ satisfying $P(n) \ge 7$, we obtain
$$
P(n) \le x^{1/\log\log x}.
$$
So, using Lemma~\ref{lem:main}, for any sufficiently large $x$ we have
\begin{equation}   \label{eq:M2}
M_2(x)  \le 12 + \Psi(x, x^{1/\log\log x})  \ll x/ \sqrt{\log x}.
\end{equation}
Finally, combining \eqref{eq:M} with \eqref{eq:M1} and \eqref{eq:M2}, we obtain
$$
M(x) \ll x/ \sqrt{\log x}.
$$
This completes the proof.
 \end{proof}

\section*{acknowledgment}
The authors would like to thank the editor and the referees for their valuable comments.
The first author was supported by the Scientific Research Foundation of Qufu Normal University No. BSQD20130139,
and the second author was supported by the Australian Research Council Grant DE190100888.


\end{document}